\documentclass[oneside,english,final]{elsarticle}
\usepackage{mathrsfs}
\usepackage{amsthm}
\usepackage{amssymb}
\usepackage{amsmath}

\usepackage{listings}
\usepackage{matlab-prettifier} 

\makeatletter
\numberwithin{equation}{section}
\numberwithin{figure}{section}
\theoremstyle{plain}
\newtheorem{thm}{\protect\theoremname}
\theoremstyle{definition}
\newtheorem{defn}[thm]{\protect\definitionname}
\theoremstyle{plain}
\newtheorem{lem}[thm]{\protect\lemmaname}
\theoremstyle{plain}
\newtheorem{cor}[thm]{\protect\corollaryname}
\theoremstyle{plain}
\newtheorem{rmk}[thm]{\protect\remarkname}
\theoremstyle{plain}
\newtheorem{example}[thm]{\protect\examplename}

\usepackage{stmaryrd}
\usepackage{mathrsfs} 
\newcommand{\norm}[1]{ \interleave  #1 \, \interleave}
\usepackage{color}

\makeatother

\usepackage{babel}
\providecommand{\corollaryname}{Corollary}
\providecommand{\definitionname}{Definition}
\providecommand{\lemmaname}{Lemma}
\providecommand{\theoremname}{Theorem}
\providecommand{\remarkname}{Remark}
\providecommand{\examplename}{Example}

\newcommand{\V}{\mathcal{V}}
\newcommand{\X}{\mathcal{X}}
\newcommand{\F}{\mathbb{F}}
\newcommand{\C}{\mathbb{C}}
\newcommand{\R}{\mathbb{R}}

\makeatletter
\def\ps@pprintTitle{%
 \let\@oddhead\@empty
 \let\@evenhead\@empty
 \let\@oddfoot\@empty
 \let\@evenfoot\@empty
}
\makeatother

\begin{document}

\begin{frontmatter}
\title{Matrix best approximation in the spectral norm}

\author[1]{Vance Faber} %\ead{vance.faber@gmail.com}
\author[2]{J{\"o}rg Liesen\corref{cor1}}\ead{liesen@math.tu-berlin.de}
\author[3]{Petr Tich\'y}\ead{petr.tichy@mff.cuni.cz}

\cortext[cor1]{Corresponding author}
\address[1]{Hoquiam, WA, USA}
\address[2]{Institute of Mathematics, TU Berlin, Stra{\ss}e des 17. Juni 136, 10623 Berlin, Germany}
\address[3]{Faculty of Mathematics and Physics, Charles University, Sokolovsk{\'a} 83, 18675 Prague, Czech Republic}

\begin{abstract}
We derive, similar to Lau and Riha in~\cite{LaRi1981}, a matrix formulation of a general best approximation theorem of Singer for the special case of spectral approximations of a given matrix from a given subspace. Using our matrix formulation we describe the relation of the spectral approximation problem to semidefinite programming, and we present a simple MATLAB code to solve the problem numerically. We then obtain geometric characterizations of spectral approximations that are based on the $k$-dimensional field of $k$ matrices, which we illustrate with several numerical examples. The general spectral approximation problem is a min-max problem, whose value is bounded from below by the corresponding max-min problem. Using our geometric characterizations of spectral approximations, we derive several necessary and sufficient as well as sufficient conditions for equality of the max-min and min-max values. Finally, we prove that the max-min and min-max values are always equal when we ``double'' the problem. Several results in this paper generalize results that have been obtained in the convergence analysis of the GMRES method for solving linear algebraic systems. 
\end{abstract}

\begin{keyword}
matrix approximation problems \sep best approximation problems \sep iterative methods \sep field of values \sep convergence analysis \sep ideal GMRES \sep worst-case GMRES  
\MSC[2020] 41A50\sep 49K35\sep 65F10 
%\sep \cblue{90C22}
\end{keyword}

\end{frontmatter}

%\maketitle

\section{Introduction }

An important problem that occurs in numerous different contexts of mathematics and its applications is the best approximation of a given element in a normed vector space by elements of a given subspace. To formulate this problem precisely, let $\V$ be a real or complex vector space endowed with the norm $\|\cdot\|$, and let $\X\subset\V$ be a subspace. Given a vector $y\in\V\setminus\X$, we then look for a vector $x_*\in\X$ that satisfies
\begin{equation}\label{eqn:gen-approx}
\|y-x_*\|=\inf_{x\in\X} \|y-x\|.
\end{equation}
Such a vector $x_*$ is called a \emph{best approximation} of $y\in\V$ from the subspace $\X$ with respect to the given norm. The general theory of best approximation in normed vector spaces is described, for example, in Singer's classical monograph published in Springer's \emph{Grundlehren Series} in 1970~\cite{Sin70}; see also his follow-up monograph from 1974~\cite{Sin74}.  

In the area of numerical linear algebra, or matrix computations, the vector space of interest usually is the $n^2$-dimensional matrix space $\V=\F^{n\times n}$, where $\F\in\{\R,\C\}$, and $\X=\mathrm{span}\{X_1,\dots,X_k\}\subset\F^{n\times n}$ is a given $k$-dimensional subspace. Because of the finite-dimensionality, the infimum in \eqref{eqn:gen-approx} can be replaced by the minimum, i.e., the best approximation problem has at least one solution.

A common norm, which we also consider in this paper, is the \emph{spectral norm}, or \emph{matrix 2-norm}, defined by
\begin{equation}\label{eqn:spectral}
\|A\|_2\equiv \max_{\substack{v\in \F^n\\ \|v\|_2=1}}\|Av\|_2,    
\end{equation}
where $\|v\|_2=(v^Hv)^{1/2}$ is the Euclidean norm on $\F^n$. A best approximation  $X_*\in\X$ of $Y\in\V\setminus\X$ with respect to the spectral norm is also called a \emph{spectral approximation} of $Y$ from the subspace $\X$. We point out that in general the spectral approximation of a matrix from a subspace need not be unique, since the spectral norm is not strictly convex. Further general details about this can be found, e.g., in~\cite{Zi1993,Zi1995}; see also Example~\ref{ex:non-unique} below. 

Frequently studied special cases of \eqref{eqn:gen-approx} for a given matrix $A\in\F^{n\times n}$ with minimal polynomial degree larger than~$k$ are the \emph{ideal GMRES problem}
\begin{align}
\|I-X_*\|_2&=\min_{X\in\X} \|I-X\|_2\quad\mbox{with}\quad \X=\mathrm{span}\{A,\dots, A^k\},\label{eqn:iG}
\end{align}
and the \emph{ideal Arnoldi problem}
\begin{align}\label{eqn:iA}
\|A^k-X_*\|_2 &=\min_{X\in\X} \|A^k-X\|_2\quad\mbox{with}\quad \X=\mathrm{span}\{I,\dots, A^{k-1}\}. 
\end{align}
These problems were introduced by Greenbaum and Trefethen in~\cite{GrTr1994}, and they have been considered, e.g., in~\cite{ChoGre15,FaJoKnMa1996,FabLieTic09,FabLieTic13,LieTic14,TiLiFa2007,To1997}. As shown in~\cite{GrTr1994}, both problems have a unique solution if $A$ is nonsingular and the optimal value is positive. These uniqueness results were extended to more general classes of best approximation problems in the spectral norm in~\cite{LiTi2009}. The estimation of the values of the two problems is closely related to the \emph{polynomial numerical hull} of matrices, originally defined by Nevanlinna in 1993~\cite[p.~42]{Nev93}, and later studied, e.g., in~\cite{BurGre06,DavLiSal08,FabGreMar03,Gre02,Gre03}. This in turn is relevant in the analysis of spectral sets and the famously still open Crouzeix Conjecture; see~\cite{Cro04,CroPal17,Gr2023} as well as the recent paper~\cite{CheGreTro25} and the references therein. 

Naturally, any general characterization of best approximations from a subspace, i.e., of solutions of the problem \eqref{eqn:gen-approx}, applies also to the special case $\V=\F^{n\times n}$ and the spectral norm. An interesting special case that is particularly useful for matrix computations has been studied by Lau and Riha in~\cite{LaRi1981}. They worked out a matrix formulation of a general best characterization theorem given by Singer in~\cite[pp.~170--171]{Sin70} (see also~\cite[pp.~7--8]{Sin74}) for the spectral approximations of a given matrix $Y\in\R^{n\times n}$ from a given subspace $\X\subset\R^{n\times n}$. This matrix formulation is based on characterizing extremal points on the unit sphere of the dual space $\V^*$. 

This paper is organized as follows. We first derive, in Section~\ref{sec:singer}, a matrix formulation of Singer's general theorem for the special case of spectral approximations of a given matrix from a given subspace. Our approach is similar to the one of Lau and Riha in~\cite{LaRi1981}, but we consider subspaces $\X\subset\F^{n\times n}$ with $\F\in\{\R,\C\}$. We also avoid the Lemma in~\cite[p.~113]{LaRi1981} (which, as pointed out by Zi{\c{e}}tak~\cite{Zi1988}, is not correct) and thus modify the proofs of ~\cite[Theorems~1 and~2]{LaRi1981}. Based on our matrix formulation we then describe the relation of the spectral approximation problem and semidefinite programming, and we present a simple MATLAB code that uses the CVX toolbox \cite{cvx} to solve the problem numerically. Using our matrix formulation we next obtain, in Section~\ref{sec:geometry}, geometric characterizations of spectral approximations that are based on the $k$-dimensional field of $k$ matrices; see Definition~\ref{def:FofM}. We illustrate our characterizations with several numerical examples. The general spectral approximation problem is a min-max problem, whose value is bounded from below by the corresponding max-min problem; see the inequality in \eqref{eqn:maxmin-minmax}. Using our geometric characterizations of spectral approximations we derive, in Section~\ref{sec:minmax}, several necessary and sufficient as well as sufficient conditions for equality of the max-min and min-max values. Our conditions generalize several known results that have been obtained in the analysis of the worst-case behavior of the GMRES method. Finally, we prove that the max-min and min-max values are always equal when we ``double'' the problem. 

\section{A matrix formulation of Singer's theorem}\label{sec:singer}

We start with some notation required to formulate Singer's theorem.

\begin{defn}\label{def1}
If $\V$ is a vector space over a field $\F$, where $\F\in\{\R,\C\}$, endowed with the norm $\|\cdot\|$, the set $\Omega\equiv \{v\,:\, v\in\V,\|v\|=1\}$ is called the \emph{unit sphere} in $\V$. Let $\V^*$ be the dual space of $\V$, i.e., the space of all linear functionals $f:\V\to\F$. We denote the \emph{dual norm} of $\|\cdot\|$ on $\V^*$ by
$$\norm{f}\equiv\sup_{v\in\Omega}|f(v)|,$$
and the unit sphere in $\V^*$ by $\Omega^*\equiv \{f\,:\, f\in\V^{*}, \norm{f}=1\}$. A functional $f\in\Omega^*$ is called \emph{extremal}, if $f=\frac12 (f_1+f_2)$ with $f_1,f_2\in\Omega^*$ implies that $f=f_1=f_2$.
\end{defn}

The following theorem of Singer~\cite[pp. 170--171]{Sin70} (see also~\cite[Theorem~1.12]{Sin74}) characterizes best approximations of a given vector from a given subspace. Singer's original formulation contains four equivalent conditions for a best approximation. Here we present only the one that is most useful for our context (namely condition (3) in his theorem).

\begin{thm}\label{thm:singer}
Let $\V$ be a vector space over a field $\F$, where $\F\in\{\R,\C\}$, endowed with the norm $\|\cdot\|$, let $\X\subset\V$ be a $k$-dimensional subspace, and let $y\in\V\setminus\X$ and $x_*\in\X$ be given. Then the following are equivalent:
\begin{itemize}
    \item[$(1)$] The vector $x_*$ is a best approximation of $y$ from $\X$ with respect to the given norm, i.e., it satisfies \eqref{eqn:gen-approx}.
    \item[$(2)$] There exist $\ell$ extremal vectors $f_{1},\dots,f_{\ell}\in\Omega^{*}$, where $1\leq\ell\leq k+1$ if $\F=\R$ and $1\leq\ell\leq2k+1$ if $\F=\C$, and $\ell$ positive real numbers $\omega_{1},\dots,\omega_{\ell}$ with $\omega_{1}+\dots+\omega_{\ell}=1$, such that 
    \begin{align*}
    \sum_{j=1}^{\ell}\omega_{j}f_{j}(x) & = 0,\quad \mbox{for all $x\in\X$, and }\\
    f_{j}(y-x_{*}) & = \|y-x_{*}\|,\quad \mbox{for $j=1,\dots,\ell$.}
    \end{align*}
\end{itemize}
\end{thm}

Note that if in condition (2) we define the functional $f\equiv \sum_{j=1}^\ell \omega_j f_j\in\V^*$, which is a convex combination of the extremal vectors $f_1,\dots,f_\ell\in\Omega^*$, then $f(y-x_*)=\|y-x_*\|$.

We now want to find a formulation of Theorem~\ref{thm:singer} for the case of the vector space $\V=\F^{n\times n}$, where $\F\in\{\R,\C\}$, with the spectral norm $\|\cdot\|_2$ (see \eqref{eqn:spectral}), and a given $k$-dimensional subspace $\X=\mathrm{span}\{X_1,\dots,X_k\}\subset\V$. Thus, for a given matrix $Y\in\V\setminus\X$ we are interested in characterizing its spectral approximations $X_*$ from the subspace $\X$, i.e., the solutions $X_*\in\X$ of the best approximation problem
\begin{equation}\label{eqn:main-prob}
\|Y-X_*\|_2=\min_{X\in\X}\|Y-X\|_2=\min_{X\in\mathcal{X}}\max_{\substack{v\in\F^n\\\|v\|_2=1}}\|(Y-X)v\|_2.
\end{equation}

In order to do so, we will determine $\norm{\cdot}_2$, the dual norm of $\|\cdot\|_2$ on $\V^*=(\F^{n\times n})^*$, and the extremal vectors in $\Omega^*\subset\V^*$.

Let $f\in\V^*$ be given. Using the standard basis $\{E_{ij}=e_ie_j^T\,:\,i,j=1,\dots,n\}$ of $\V$, every matrix $A=[a_{ij}]\in \V$ can be written as $\sum_{i,j=1}^n a_{ij}E_{ij}$. Thus,
$$f(A)=f\Big(\sum_{i,j=1}^n a_{ij}E_{ij}\Big)=\sum_{i,j=1}^n a_{ij}f(E_{ij})=\mathrm{trace}(F^{H}A),$$
where 
$$F=[f_{ij}]\in\V\quad\mbox{with}\quad f_{ij}\equiv \overline{f(E_{ij})}.$$ 
We call the uniquely determined matrix $F\in\V$ the \emph{standard matrix representation} of the linear functional $f\in\V^*$. It is well known that the map $\V^*\to \V$, $f\mapsto F$, is a vector space isomorphism.

\begin{rmk}
Note that, in general, a matrix representation of a linear functional $f:\V\rightarrow\F$ depends on the chosen bases of $\V$ and $\F$. Our choice of the two bases, namely the standard basis of $\V$ and the basis $\{1\}$ of $\F$, yields the unique matrix representation $F$ satisfying $f(A)=\mathrm{trace}(F^HA)$ for each $A\in\V$. This feature is convenient for the derivation of the dual norm $\norm{\cdot}_2$ on $\V^*$.
\end{rmk}

If we denote
$$\langle F,A\rangle\equiv \mathrm{trace}(F^{H}A),$$
then the dual norm of $\|\cdot\|_2$ on $\V^*$ is given by (cf. Definition~\ref{def1})
\begin{equation}\label{eqn:dual}
\norm{f}_2=\max_{A\in\Omega}|f(A)|=\max_{A\in\Omega}|\langle F,A\rangle|.
\end{equation}
Suppose that $\mathrm{rank}(F)=r\leq n$, and let 
\begin{equation}\label{eqn:svd}
F=U_r\Sigma_r V_r^H= \sum_{j=1}^r \sigma_j u_jv_j^H,\quad \sigma_1\geq\cdots\geq\sigma_r>0,
\end{equation}
be a singular value decomposition of $F$. Then for any $A=[a_{ij}]\in \V$ we have 
\begin{align}\label{eqn:ident}
\langle F,A\rangle &=
\mathrm{trace}(V_r\Sigma_r U_r^HA)=
\mathrm{trace}(\Sigma_r U_r^HAV_r)=\sum_{j=1}^r \sigma_j u_j^H Av_j.
\end{align}
We use this identity to prove the following result; cf.~\cite[Theorem~1]{LaRi1981}.

\begin{lem}
Consider the vector space $\F^{n\times n}$ with the spectral norm $\|\cdot\|_2$. If $f\in(\F^{n\times n})^*$ has the standard matrix representation $F\in\F^{n\times n}$ with the singular value decomposition \eqref{eqn:svd}, then 
$$\norm{f}_2=\sum_{j=1}^r\sigma_j.$$
Thus, $\norm{f}_2$ is equal to the Schatten 1-norm (or nuclear norm, trace norm) of $F$, and we write $\norm{F}_2\equiv \norm{f}_2$.
\end{lem}

\begin{proof}
The definition of the dual norm in \eqref{eqn:dual}, the identity \eqref{eqn:ident}, and the triangle and Cauchy-Schwarz inequalities yield
$$\norm{f}_2\leq \sum_{j=1}^r \sigma_j \max_{A\in\Omega} |u_j^H Av_j|\leq 
\sum_{j=1}^r \sigma_j \max_{A\in\Omega} \|Av_j\|_2 =\sum_{j=1}^r \sigma_j.$$
On the other hand, for the matrix $A=U_rV_r^H\in\Omega$ we obtain
$$\norm{f}_2\geq |\langle F,A\rangle|=\mathrm{trace}(V_r\Sigma_r U_r^HU_r V_r^H)=\mathrm{trace}(\Sigma_r)=\sum_{j=1}^r \sigma_j.$$
\end{proof}

Note that any matrix $F\in\F^{n\times n}$ with a singular value decomposition of the form \eqref{eqn:svd} satisfies 
$$\sigma_1=\|F\|_2\leq \norm{F}_2=\sum_{j=1}^r \sigma_j,$$
which shows that $\|F\|_2=\norm{F}_2$ holds if and only if $\mathrm{rank}(F)=1$. 
The next lemma is~\cite[Theorem~2]{LaRi1981}, here including also the complex case and with a modified proof.

\begin{lem}
Consider the vector space $\F^{n\times n}$ with the spectral norm $\|\cdot\|_2$ and the corresponding unit sphere $\Omega$. A functional $f\in\Omega^*$ is extremal if and only if its standard matrix representation $F\in\F^{n\times n}$ is of the form $F=uv^H$ with unit Euclidean norm vectors $u,v\in\F^n$.     
\end{lem}

\begin{proof}
We start analogously to the proof in~\cite[p.~115]{LaRi1981}. Suppose that $f\in\Omega^*$ has the standard matrix representation $F\in\F^{n\times n}$ with $\mathrm{rank}(F)=r\geq 2$. We will show that $f$ is then not extremal. We have $1=\norm{f}_2=\norm{F}_2$, and hence $F$ has a singular value decomposition of the form
$$F=\sum_{j=1}^r \sigma_j u_jv_j^H,\quad \mbox{where}\quad \sum_{j=1}^r \sigma_j=1.$$
In particular, we must have $0<\sigma_2\leq \sigma_1<1$. 
Let $\epsilon>0$ be such that $\sigma_1+\epsilon<1$ and $\sigma_2-\epsilon>0$, and define the matrices
\begin{align*}
F_1 &\equiv (\sigma_1+\epsilon)u_1v_1^H+(\sigma_2-\epsilon)u_2v_2^H+\sum_{j=3}^r\sigma_j u_jv_j^H,\\
F_2 &\equiv (\sigma_1-\epsilon)u_1v_1^H+(\sigma_2+\epsilon)u_2v_2^H+\sum_{j=3}^r\sigma_j u_jv_j^H.
\end{align*}
Then $\norm{F_1}_2=\norm{F_2}_2=1$, and $F=\frac12(F_1+F_2)$, while $F_1\neq F_2$. 
If $f_1,f_2\in\Omega^*$ are the uniquely determined linear functionals corresponding to the matrices $F_1,F_2$, respectively, then $f=\frac12(f_1+f_2)$, while $f_1\neq f_2$, which shows that $f$ is not extremal.

Now let $f\in\Omega^*$ have the standard matrix representation $F\in\F^{n\times n}$ with $\mathrm{rank}(F)=1$. Then $\norm{f}_2=\norm{F}_2=\|F\|_2=1$, and hence $F=uv^H$ with $\|u\|_2=\|v\|_2=1$. We will show that $f$ is extremal, but we proceed differently than in~\cite{LaRi1981}.

Let $f_1,f_2\in\Omega^*$ be such that $f=\frac12(f_1+f_2)$, then the corresponding standard matrix representations satisfy $F=\frac12(F_1+F_2)$. Since $\norm{F_1}_2=\norm{F_2}_2=1$, we have $\|F_1\|_2\leq 1$ and $\|F_2\|_2\leq 1$, so that
$$1=\norm{F}_2=\|F\|_2=\frac12 \|F_1+F_2\|_2\leq \frac12 (\|F_1\|_2+\|F_2\|_2)\leq 1,$$
which implies $\|F_1\|_2=\|F_2\|_2=1$. 

We now consider one of the matrices $F_j$, $j=1,2$. From $\norm{F_j}_2=\|F_j\|_2=1$ we see that $\mathrm{rank}(F_j)=1$, and hence $F_j=u_jv_j^H$ with $\|u_j\|_2=\|v_j\|_2=1$. Since $\|F_j\|_2=1$, we have $|u^HF_jv|\leq 1$. The equalities $F=\frac12(F_1+F_2)$ and $F=uv^H$ imply $2=u^HF_1v+u^HF_2v$, and combining this with $|u^HF_jv|\leq 1$ yields 
$$1=u^HF_jv=(u^Hu_j)(v_j^Hv),$$
from which we obtain $u^Hu_j=e^{i\alpha_j}$ and $v_j^Hv=e^{-i\alpha_j}$ for some $\alpha_j\in [0,2\pi)$. Since $1=|u^Hu_j|=\|u\|_2\|u_j\|_2$, we must have $u_j=\beta_j u$ with $|\beta_j|=1$, and thus $\beta_j=e^{i\alpha_j}$, i.e., $u_j=e^{i\alpha_j}u$. Analogously we get $v_j=e^{i\alpha_j}v$, and hence
$$F_j=u_jv_j^H=(e^{i\alpha_j} u)(e^{i\alpha_j} v)^H=uv^H=F, \quad j=1,2.$$
This yields $f_1=f_2=f$, so that $f$ is indeed extremal.
\end{proof}

If $f\in\Omega^*$ is extremal, and $F=uv^H$ is its standard matrix representation, then for every $A\in\F^{n\times n}$ we have
$$\langle F,A\rangle =\mathrm{trace}(uv^HA)=\mathrm{trace}(v^HAu)=v^HAu.$$
We use this observation in condition (2) of our following matrix-version of Theorem~\ref{thm:singer}; cf.~\cite[Theorem~3]{LaRi1981}.

\begin{thm}\label{thm:singer-mat1}
Consider a $k$-dimensional subspace $\X\subset \F^{n\times n}$, and let $Y\in\F^{n\times n}\setminus\X$ and $X_*\in\X$ be given. Then the following are equivalent:
\begin{itemize}
    \item[$(1)$] The matrix $X_*$ is a spectral approximation of $Y$ from the subspace $\X$, i.e., it solves the problem \eqref{eqn:main-prob}.
    \item[$(2)$] There exist $\ell$ rank-one matrices $u_1v_1^H,\dots,u_\ell v_\ell^H\in\F^{n\times n}$, with $\|u_i\|_2=\|v_i\|_2=1$ for $i=1,\dots,\ell$, where $1\leq\ell\leq k+1$ if $\F=\R$ and $1\leq\ell\leq2k+1$ if $\F=\C$, and $\ell$ positive real numbers $\omega_{1},\dots,\omega_{\ell}$ with $\omega_{1}+\dots+\omega_{\ell}=1$, such that 
    \begin{align*}
    \sum_{j=1}^{\ell}\omega_{j}u_j^H Xv_j & = 0,\quad \mbox{for all $X\in\X$, and }\\
    u_j^H(Y-X_*)v_j & = \|Y-X_{*}\|_2,\quad \mbox{for $j=1,\dots,\ell$.}
    \end{align*}
\end{itemize}
\end{thm}

Further equivalent characterizations of spectral approximations are given in~\cite[Corollary~4.1]{Zi1993} and~\cite[p.~173]{Zi1996} (cf. also~\cite[Lemma~3.2]{LiTi2009} and~\cite[Section~7]{Zi2017}).

If we denote the rank-one matrices in condition (2) of Theorem~\ref{thm:singer-mat1} by $F_j\equiv u_jv_j^H$, $j=1,\dots,\ell$, and define $F\equiv \sum_{j=1}^\ell \omega_j F_j$, then $F\in\F^{n\times n}$ is a matrix with $\norm{F}_2=1$ and rank at most $\ell$ that satisfies
$$\langle F, Y-X_*\rangle = \|Y-X_*\|_2.$$
This observation is a variant of~\cite[Theorem~4.1]{Zi1993}. 

Next note that if
\[
u_j^{H}(Y-X_{*})v_j=\|Y-X_{*}\|_2
\]
holds for unit Euclidean norm vectors $u_j$ and $v_j$, then these vectors are left resp. right singular vectors corresponding to the maximal singular value of $Y-X_{*}$. We call such vectors \emph{maximal} left and right singular vectors, and we denote by $\Sigma_{Y-X_{*}}$ the span of the maximal right singular vectors of $Y-X_{*}$. We thus can formulate the following result.

\begin{cor} 
\label{cor:cond3}
In the notation of Theorem~\ref{thm:singer-mat1}, the matrix $X_*\in\X$ solves the problem \eqref{eqn:main-prob} if and only if there exist $\ell$ vectors $v_1,\dots,v_\ell\in\Sigma_{Y-X_*}$, where $1\leq\ell\leq k+1$ if $\F=\R$ and $1\leq\ell\leq2k+1$ if $\F=\C$, and $\ell$ positive real numbers $\omega_{1},\dots,\omega_{\ell}$ with $\omega_{1}+\dots+\omega_{\ell}=1$, such that 
\begin{equation}\label{eqn:cond3}
\sum_{j=1}^{\ell}\omega_{j}u_j^H Xv_j = 0,\quad \mbox{for all $X\in\X$,}
\end{equation}
where the $u_j$ are corresponding maximal left singular vectors, i.e., $(Y-X_*)v_j=\sigma_1 u_j$ for $j=1,\dots,\ell$.
\end{cor}

The problem of minimizing the spectral norm of the matrix $Y-X$ also can be formulated as a semidefinite program; see \cite[p.~55]{VaBo1996} for the case $\F=\R$ and \cite{ToTr1998} for $\F=\C$.
The notion of duality discussed in this section via Singer's theorem
then coincides with semidefinite programming duality; see \cite[pp.~67--68]{VaBo1996} for $\F=\R$. Based on the results of \cite{ToTr1998},
extending the dual problem formulation to $\mathbb{F}=\mathbb{C}$
would be straightforward. Since the optimal duality gap is zero for this
class of semidefinite programs, it becomes possible to investigate
the structure of the Hermitian positive semidefinite matrices in $\mathbb{F}^{2n\times2n}$
that solve the semidefinite dual problem. In particular, this structure can be described in terms of the maximal right and left singular vectors of $Y-X_{*}$ and the positive scalars $\omega_{i}$ introduced above. This would offer yet another equivalent characterization of a spectral approximation $X_{*}$ from the semidefinite programming perspective. 
The precise formulation of this characterization lies beyond the scope of this paper, but we will mention some details for the case $\F=\R$.

In the notation of Corollary~\ref{cor:cond3} we can define the matrices 
$$U\equiv [u_1,\dots,u_\ell],\quad V\equiv [v_1,\dots,v_\ell],\quad
\Lambda\equiv \frac12 {\rm diag}(\omega_1,\dots,\omega_\ell),$$
and
$$
Z\equiv \begin{bmatrix}
    U\Lambda U^T & U\Lambda V^T\\ V\Lambda U^T & V\Lambda V^T
\end{bmatrix}\equiv \begin{bmatrix} C & B^T\\ B & D\end{bmatrix}.
$$
Then $Z\in\R^{2n\times 2n}$ is a symmetric positive semidefinite matrix with $\mathrm{trace}(Z)=1$, and according to \eqref{eqn:cond3}  we have
$$
\mathrm{trace}(BX_i)=\mathrm{trace}(\Lambda U^TX_iV)=\sum_{j=1}^\ell\omega_j u_j^TX_i v_j=0, \quad i=1,\dots,k.
$$
Moreover,
\begin{align*}
\mathrm{trace}(BY) &=
\mathrm{trace}(B(Y-X_*)) = 
\mathrm{trace}(\Lambda U^T(Y-X_*)V) = 
\frac{\sigma_1}{2},
\end{align*}
so that $Z$ solves the dual problem 
\begin{equation}
\max2\,\mathrm{trace}\left(BY\right)\quad\mbox{s.t.}\quad Z^{T}=Z\equiv\left[\begin{array}{cc}
C & B^{T}\\
B & D
\end{array}\right]\geq 0,\label{eq:dual}
\end{equation}
\[
0=\mathrm{trace}\left(BX_{i}\right),\ i=1,\dots,k,\quad1=\mathrm{trace}(Z);
\]
see \cite[p.~66]{VaBo1996} for the formulation of the dual semidefinite program to \eqref{eqn:main-prob}. (Here $Z\geq 0$ means that $Z$ is positive semidefinite.)

Note that formulating \eqref{eqn:main-prob} as a semidefinite program is also useful for computing spectral approximations numerically. To solve problem \eqref{eqn:main-prob} one can use CVX, a package for specifying and solving convex programs \cite{cvx,gb08}.
The standard CVX distribution includes well-tuned software packages for semidefinite programming such as SDPT3~\cite{SDPT3} and SeDuMi~\cite{SeDuMi}. For completeness we provide the MATLAB function {\tt bestmat}, which we have used in Example~\ref{ex:LauRhia} to compute a spectral approximation.

\begin{lstlisting}[style=Matlab-editor] 
function [S,R,a] = bestmat(Y,X)
% INPUT
%   Y ... the n x n matrix to be approximated
%   X ... a basis of the k-dimensional subspace
% OUTPUT
%   S ... a spectral approximation X_* of Y
%   R ... the residual matrix Y - S
%   a ... the coefficients of S in the basis X
n = size(Y, 1); 
k = length(X);
cvx_precision best 
cvx_begin sdp
    variables a(k) t
    expression R(n, n)
    S = zeros(n,n);
    for i=1:k, S = S + a(i)*X{i}; end
    R = Y - S;
    minimize( t )
    subject to
        [t * eye(n), R;
         R', t * eye(n)] >= 0;
cvx_end
end 
\end{lstlisting}

\section{Geometric characterizations of spectral approximations}\label{sec:geometry}

In this section we derive further equivalent characterizations of spectral approximations. Our approach is motivated by the analysis in~\cite{FaJoKnMa1996}, where the authors used certain generalizations of the field of values of matrices in the convergence analysis of Krylov subspace methods for solving linear algebraic systems; see Section~\ref{sec:minmax} for further details about this. In order to proceed we need the following definition; see also \cite[Definition 1.8.10, p.~85]{B:HoJo1994}.

\begin{defn}\label{def:FofM}
Let $k\geq 1$ be a given integer, and let $\F\in\{\R,\C\}$. The {\em $k$-dimensional field of the $k$ matrices $A_1,\dots,A_k\in{\mathbb F}^{n\times n}$ restricted to the subspace $\mathcal S\subseteq \F^n$} is defined by
\begin{equation}\label{eqn:GFV}
\mathscr{F}(A_1,\dots,A_k;{\mathcal S})
\equiv
\left\{\begin{bmatrix}
v^{H}A_{1}v\\
\vdots\\
v^{H}A_{k}v
\end{bmatrix}: v\in{\mathcal S},\;\|v\|_2=1\right\}\,\subset\,\F^k.    
\end{equation}
If $S=\F^n$, we call the set $\mathscr{F}(A_1,\dots,A_k;\F^n)$ the {\em $k$-dimensional field of the $k$ matrices $A_1,\dots,A_k$}, and denote it simply by $\mathscr{F}(A_1,\dots,A_k)$.
\end{defn}

Note that the matrices $A_1,\dots,A_k$ and the subspace ${\mathcal S}$ in Definition~\ref{def:FofM} are considered over the same field ${\mathbb F}$. Thus, if we consider $A_1,\dots,A_k \in \R^{n\times n}$, then $\mathcal{S}\subseteq \R^n$. In particular, the classical field of values of a single matrix $A\in\F^{n\times n}$ is the set $\mathscr{F}(A)=\mathscr{F}(A;{\C^n})$ (here a real matrix $A$ is considered complex). 

The set $\mathscr{F}(A_1,\dots,A_k)$ was called \emph{the generalized field of values} of the matrices $A_1,\dots,A_k\in\F^{n\times n}$ in~\cite{FaJoKnMa1996}. The set $\mathscr{F}(A_1,\dots,A_k;\C^n)$ for $A_1,\dots,A_k\in\F^{n\times n}$ was called the \emph{joint numerical range}, e.g., in~\cite{BiLi91,DavLiSal08,GutJonKar04,LauLiPoo22,LiPoo99}, the \emph{multiform numerical range} in~\cite{Poo94}, and the \emph{$k$-dimensional generalized field of values} in~\cite{Gre02}. The set $\mathscr{F}((A-\zeta I),\dots,(A-\zeta I)^k;\C^n)$ for certain $\zeta\in\C$ is useful in the analysis of the polynomial numerical hull of degree $k$ of a given matrix $A\in\F^{n\times n}$; see, e.g.,~\cite{DavLiSal08,Gre02}.

Let us give two illustrations of the $k$-dimensional field of $k$ matrices with a particular focus on the (non-)convexity if this set, which will be of great interest in the following.

\begin{example} (cf. \cite[Example~3.1]{MulTom20}) 
We consider $\F=\R$ and the two matrices
$$A_1=\begin{bmatrix} 1& 0\\ 0 & 0\end{bmatrix},\quad
A_2=\begin{bmatrix} 0& 0\\ 1 & 0\end{bmatrix}$$
in $\R^{2\times 2}$. Writing the elements of $\R^2$ in the form $v=[z_1,z_2]^T$ with $z_1,z_2\in\R$ we obtain
$$\mathscr{F}(A_1,A_2)=\left\{
\begin{bmatrix}
z_1^2\\ z_1z_2\end{bmatrix} : z_1,z_2\in\R,\;z_1^2+z_2^2=1\right\}\subset\R^2.$$
We now easily see that
$$\begin{bmatrix}
    0\\0
\end{bmatrix},
\begin{bmatrix}
    1\\0
\end{bmatrix}\in \mathscr{F}(A_1,A_2),\quad\mbox{but}\quad
\begin{bmatrix}
    z\\0
\end{bmatrix}\notin \mathscr{F}(A_1,A_2)\;\mbox{for $0<z<1$,}$$
and hence $\mathscr{F}(A_1,A_2)$ is not convex.
\end{example}

\begin{example} (cf.~\cite[Example~1.2]{BiLi91} or~\cite[Example~1]{GutJonKar04})
We consider $\F=\C$ and the three matrices
$$A_1=\begin{bmatrix} 1& 0\\ 0 & -1\end{bmatrix},\quad
A_2=\begin{bmatrix} 0& 1\\ 1 & 0\end{bmatrix},\quad
A_3=\begin{bmatrix} 0& -i\\ -i & 0\end{bmatrix}=-i A_2$$
in $\C^{2\times 2}$. Writing the elements of $\C^2$ in the form $v=[z_1,z_2]^T$ with $z_1,z_2\in\C$, we obtain
$$\mathscr{F}(A_1,A_2,A_3)=\left\{
\begin{bmatrix}
|z_1|^2-|z_2|^2 \\ 2\mathrm{Re}(\overline{z}_1z_2) \\ 
2\mathrm{Im}(\overline{z}_1z_2) \end{bmatrix} : z_1,z_2\in\C,\;|z_1|^2+|z_2|^2=1\right\}\subset\R^3.$$
Every element $w\in\mathscr{F}(A_1,A_2,A_3)$ satisfies
\begin{align*}
\|w\|_2^2 =(|z_1|^2-|z_2|^2)^2+4|z_1z_2|^2=(|z_1|^2+|z_2|^2)^2=1,   
\end{align*}
i.e., $\mathscr{F}(A_1,A_2,A_3)$ is the unit sphere in $\R^3$. Thus, the set $\mathscr{F}(A_1,A_2,A_3)$ is real (although $\F=\C$) and not convex. 
\end{example}

We are now ready to formulate another equivalent characterization of 
the solution of the problem \eqref{eqn:main-prob}.

\begin{cor} \label{cor:cond4}
Consider a $k$-dimensional subspace $\X=\mathrm{span}\{X_1,\dots,X_k\}\subset \F^{n\times n}$, and let $Y\in\F^{n\times n}\setminus\X$ and $X_*\in\X$ be given.
The matrix $X_*\in\X$ solves the problem \eqref{eqn:main-prob} if and only if 
\begin{equation}\label{eqn:cond4}
0 \in 
\mathrm{cvx}\left( \mathscr{F}(A_1,\dots,A_k;\Sigma_{Y-X_{*}}) \right),
\end{equation}
where $\mathrm{cvx}$ denotes the convex hull, and 
\begin{equation}\label{eqn:Ai}
    A_i \equiv (Y-X_{*})^{H}X_{i},\quad i=1,\dots,k.
\end{equation}
\end{cor}

\begin{proof} Using $(Y-X_*)v_i=\sigma_1 u_i$, 
 condition \eqref{eqn:cond3} in Corollary~\ref{cor:cond3} is equivalent to
\[
\sum_{j=1}^{\ell}\omega_j \begin{bmatrix}
v_{j}^{H}A_{1}v_{j}\\
\vdots\\
v_{j}^{H}A_{k}v_{j}
\end{bmatrix}=
 \begin{bmatrix}
0\\
\vdots\\
0
\end{bmatrix}
 \in\mathbb{F}^{k}.
\]
Since $\omega_{1},\dots,\omega_{\ell}$ are positive real numbers with $\omega_{1}+\dots+\omega_{\ell}=1$, this means that the zero vector lies in the convex hull of the set $ \mathscr{F}(A_1,\dots,A_k;\Sigma_{Y-X_{*}})$.
\end{proof}
\medskip

The characterization presented in Corrolary~\ref{cor:cond4} can be identified with \cite[Lemma~2.15]{FaJoKnMa1996}, and also with \cite[Lemma 2.4]{GrGu1994} (for real matrices). While the proofs in~\cite{FaJoKnMa1996,GrGu1994} are non-trivial, our proof easily follows from the charaterzation of the spectral approximation in Theorem~\ref{thm:singer-mat1} and its reformulation in Corollary~\ref{cor:cond3}.

The characterization in Corollary~\ref{cor:cond4} uses the span of the maximal right singular vectors $\Sigma_{Y-X_*}$, which is inconvenient from a practical point of view. We next derive yet another equivalent characterization that avoids using this span.

\begin{thm}\label{thm:cond5}
Consider a $k$-dimensional subspace $\X=\mathrm{span}\{X_1,\dots,X_k\}\subset\F^{n\times n}$, and let $Y\in\F^{n\times n}\setminus\X$ and $X_*\in\X$ be given. In the notation of Corollary~\ref{cor:cond4}, 
condition \eqref{eqn:cond4} holds if and only if 
\begin{equation}\label{eqn:cond5}
\varrho \, e_1 \in 
\mathrm{cvx} \left( \mathscr{F}(A_Y, A_1,\dots,A_k) \right),
\end{equation}
where the matrices $A_1,\dots,A_k$ are defined by \eqref{eqn:Ai}, 
$$
    A_Y \equiv (Y-X_*)^HY,\quad \varrho \equiv \|Y-X_{*}\|_2^{2},
$$
and $e_1 \in \R^{k+1}$ is the first column of the identity matrix.
\end{thm}

\begin{proof}
If $0\in\mathrm{cvx}\left( \mathscr{F}(A_1,\dots,A_k;\Sigma_{Y-X_{*}}) \right)$, then there exist $v,w\in\Sigma_{Y-X_{*}}$ and 
$\alpha\in [0,1]$, such that
\[
\alpha\begin{bmatrix}
w^{H}A_1w\\
\vdots\\
w^{H}A_{k}w
\end{bmatrix}+(1-\alpha)\begin{bmatrix}
v^{H}A_{1}v\\
\vdots\\
v^{H}A_{k}v
\end{bmatrix}=\begin{bmatrix}
0\\
\vdots\\
0
\end{bmatrix}.
\]
Using $X_*\in \mathrm{span}\{X_1,\dots,X_k\}$ and denoting 
$A_{*}\equiv (Y-X_{*})^{H}X_{*}$, we find out that
$A_{*} \in \mathrm{span}\{A_1,\dots,A_k\}$ and, therefore also
$$
\alpha w^{H}A_{*}w+(1-\alpha)v^{H}A_{*}v=0.
$$
Further, realize that for any $z\in\Sigma_{Y-X_{*}}$ it holds that
$$
\varrho = z^{H}\left(Y-X_{*}\right)^{H}\left(Y-X_{*}\right)z
= z^{H}(A_Y-A_{*})z.
$$
Taking any two vectors $v,w\in\Sigma_{Y-X_{*}}$ 
we obtain
\begin{align*}
\varrho  &=  \alpha w^{H}\left(A_Y-A_{*}\right)w+(1-\alpha)v^{H}\left(A_Y-A_{*}\right)v\\
  &=  \alpha w^{H}A_Y w+(1-\alpha)v^{H} A_Y v,
\end{align*}
which shows that 
\begin{equation}\label{eqn:cond5aux}
\alpha\begin{bmatrix}
w^{H}A_Yw\\
w^{H}A_1w\\
\vdots\\
w^{H}A_{k}w
\end{bmatrix}+(1-\alpha)\begin{bmatrix}
v^{H}A_Yv\\
v^{H}A_{1}v\\
\vdots\\
v^{H}A_{k}v
\end{bmatrix}=\begin{bmatrix}
\varrho \\
0\\
\vdots\\
0
\end{bmatrix},
\end{equation}
i.e., the condition~\eqref{eqn:cond5} holds.

On the other hand, let \eqref{eqn:cond5aux} hold for some unit Euclidean norm vectors $v,w\in\F^n$ and a scalar $\alpha\in [0,1]$. Since $A_*\in\mathrm{span}\{A_1,\dots,A_k\}$, the last $k$ entries in this vector equation yield
\[
\alpha w^{H}A_{*}w+(1-\alpha)v^{H}A_{*}v=0.
\]
Using this in the first entry of the above vector equation \eqref{eqn:cond5aux} leads to
\begin{align*}
\varrho & =  \alpha w^{H}A_Yw+(1-\alpha)v^{H}A_Yv\\
& =  \alpha w^{H}(A_Y-A_*)w+(1-\alpha)v^{H}(A_Y-A_*)v\\
 & =  \alpha w^{H}\left(Y-X_{*}\right)^{H}\left(Y-X_{*}\right)w+(1-\alpha)v^{H}\left(Y-X_{*}\right)^{H}\left(Y-X_{*}\right)v.
\end{align*}
This shows that $v,w\in\Sigma_{Y-X_{*}}$, and hence $0\in\mathrm{cvx}\left( \mathscr{F}(A_1,\dots,A_k;\Sigma_{Y-X_{*}}) \right)$. 
\end{proof}

We next illustrate the characterization \eqref{eqn:cond5} in Theorem~\ref{thm:cond5} using several examples of known spectral approximations that have appeared in the previous literature.

\begin{example}\label{ex:singular}
Let $\X=\mathrm{span}\{X_1,\dots,X_k\}$ be any $k$-dimensional subspace of $\F^{n\times n}$, where each matrix $X\in\X$ statisfies $Xw=0$ for some fixed unit norm vector $w\in\F^n$. Thus, the matrices in $\X$ are singular and have a common eigenvector corresponding to the common eigenvalue zero. Let $Y\in\F^{n\times n}\setminus\X$ be any matrix that satisfies $Yw=\|Y\|_2w$, i.e., the vector $w$ is a maximal right singular vector of $Y$. Using our characterization in Theorem~\ref{thm:cond5} we want to show that $X_*=0\in\X$ is a spectral approximation of $Y$ from~$\X$. (This is easily seen by the fact that any $X\in\X$ satisfies $\|Y-X\|_2\geq \|(Y-X)w\|_2=\|Y\|_2$, with equality for $X=0$.) For $X_*=0$ and the matrices used in Theorem~\ref{thm:cond5} we obtain
$$\begin{bmatrix}
w^HA_Yw\\
w^{H}A_{1}w\\
\vdots\\
w^{H}A_{k}w
\end{bmatrix}=
\begin{bmatrix}
w^HY^HYw\\
w^{H}Y^HX_1w\\
\vdots\\
w^{H}Y^HX_kw
\end{bmatrix}=
\begin{bmatrix}
\|Y\|_2^2\\
0\\
\vdots\\
0
\end{bmatrix}=\varrho\, e_1\in\mathscr{F}(A_Y,A_1,\dots,A_k),$$
and hence the condition of Theorem~\ref{thm:cond5} is satisfied. A special case of this example is given by a singular matrix $A\in\F^{n\times n}$, and $X_i=A^i$ for $i=1,\dots,k$, where $k$ is less than the degree of the minimal polynomial of $A$. Then $Y=I\in\F^{n\times n}\setminus\X$,
and we can formulate the result just obtained in the form 
$$\min_{\substack{\mathrm{deg}(p)\leq k\\p(0)=1}}\|p(A)\|_2=1,$$
i.e., ideal GMRES (see \eqref{eqn:iG}) completely stagnates for a singular matrix $A$; cf.~\cite[Proposition 2.1]{FaJoKnMa1996}.
\end{example}

\begin{example}\label{ex:non-unique}
Consider a singular diagonal matrix 
$$A=\mathrm{diag}(0,\lambda_2,\dots,\lambda_n)\in\R^{n\times n},$$
where $\lambda_2,\dots,\lambda_{n}$ are nonzero and pairwise distinct, and the subspace $\X={\rm span}\{A,\dots,A^k\}\subset\R^{n\times n}$, where $1\leq k\leq n-1$. We already know from Example~\ref{ex:singular}, that in this case $X_*=0$ is a spectral approximation of $Y=I$ from~$\X$. As explained in~\cite[Remark~2.4]{LiTi2009} (for a more general problem), the spectral approximation in this case is not uniquely determined. Indeed, any matrix 
$$X_*=\sum_{j=1}^k\alpha_j A^j \equiv p(A)\in\X$$
that satisfies $|1-p(\lambda_i)|\leq 1$ for $i=2,\dots,n$ is a spectral approximation of $Y=I$ from~$\X$. In Figure~\ref{fig:kfield-plots} (top) we show the set $\mathscr{F}(A_Y,A_1,A_2)\subset\R^3$ for the matrix $A\in\R^{3,3}$ with $\lambda_2=1$, $\lambda_3=2$, $\X=\mathrm{span}\{A,A^2\}$, and $X_*=0.2\,A+0.3\,A^2$, which yields $\varrho=\|Y-X_*\|_2^2=1$. The location of the vector $e_1\in\R^3$ is marked by a big (red) dot. We generated all plots of $\mathscr{F}(A_Y,A_1,A_2)\subset\R^3$ in Figure~\ref{fig:kfield-plots} in MATLAB R2024b using the function \texttt{sphere} for obtaining unit norm vectors $v\in\R^3$, and the function \texttt{surf} for plotting the corresponding vectors $[v^TA_Yv,\,v^TA_1v,\,v^TA_2v]^T\in\R^3$. 
\end{example}

\begin{example}\label{ex:jordan}
Let $J_\lambda\in\F^{n\times n}$ be the $n\times n$ Jordan block with eigenvalue $\lambda$. We consider the $k$-dimensional subspace $\X=\mathrm{span}\{I,J_\lambda,\dots,J_\lambda^{k-1}\}\subset \F^{n\times n}$, and $Y=J_\lambda^{k}$, where $1\leq k\leq n-1$. It was shown in~\cite[Theorem~3.4] {LiTi2009} that
$$X_*=J_\lambda^k-(J_\lambda-\lambda I)^k=Y-J_0^k$$ 
is a spectral approximation of $Y$ from~$\X$. In the notation of Theorem~\ref{thm:cond5} we have $Y-X_*=J_0^k$, hence $\varrho=\|Y-X_*\|_2^2=1$. For the vector $v=e_{k+1}$ we obtain $v^H(Y-X_*)^H=(J_0^ke_{k+1})^T=e_1^T$. Therefore,
$$\begin{bmatrix}
v^HA_Yv\\
v^{H}A_{1}v\\
\vdots\\
v^{H}A_{k}v
\end{bmatrix}=
\begin{bmatrix}
e_1^TJ_\lambda^{k}e_{k+1}\\
e_1^TJ_\lambda^{k-1}e_{k+1}\\
\vdots\\
e_1^Te_{k+1}
\end{bmatrix}= e_1\in\mathscr{F}(A_Y,A_1,\dots,A_k),$$
and hence the condition of Theorem~\ref{thm:cond5} is satisfied. In Figure~\ref{fig:kfield-plots} (middle) we show the set $\mathscr{F}(A_Y,A_1,A_2)\subset\R^3$ for the matrix $J_1\in\R^{3,3}$, and again the location of the vector $e_1\in\R^3$ is marked by a big (red) dot.
\end{example}

\begin{example}\label{ex:LauRhia}
Lau and Riha~\cite[pp.~119--120]{LaRi1981} gave a numerical example for a spectral approximation of the symmetric matrix 
$$Y=\begin{bmatrix}
0.67 & -0.13 & 0.27\\ -0.13 & 0.49 & 0.33\\ 0.27 & 0.33 & 0.63\end{bmatrix}\in\R^{3\times 3}$$
from the subspace $\X={\mathrm span}\{X_1,X_2\}\subset\R^{3,3}$, where
$$X_1 = \begin{bmatrix} 0.13 & 0.04 & -0.06\\ 0.04 & 0.44 & 0.21\\  -0.06 & 0.21 & 0.39\end{bmatrix},\quad
X_2 = \begin{bmatrix}0.22 & -0.09 & -0.11\\ -0.09 & 0.35 & 0.18\\ -0.11 & 0.18 & 0.51\end{bmatrix}.$$
According to our computation using the CVX software and the MATLAB function {\tt bestmat}, a spectral approximation of $Y$ is given by 
$$X_*=0.270019\,X_1+1.587964\,X_2,$$
where the coefficients are rounded to six decimal places. (In~\cite{LaRi1981} the coefficients rounded to six decimal places are $0.270020$ and $1.587963$.) Using MATLAB's {\tt norm} function we compute $\varrho=\|Y-X_*\|_2^2\approx 0.294068$.  In Figure~\ref{fig:kfield-plots} (bottom) we show the corresponding set $\mathscr{F}(A_Y,A_1,A_2)\subset\R^3$. The location of the vector $\varrho\, e_1\in\R^3$, again marked by a big (red) dot, numerically confirms that $\varrho\, e_1\in \mathscr{F}(A_Y,A_1,A_2)$.
\end{example}

\begin{figure}
    \centering
    \includegraphics[width=0.67\linewidth]{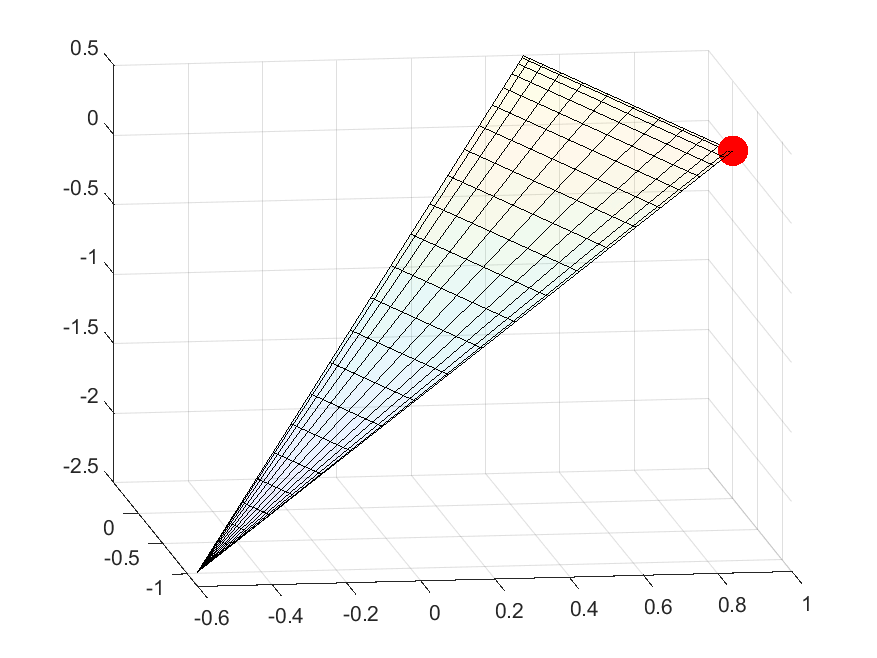}\\
    \includegraphics[width=0.67\linewidth]{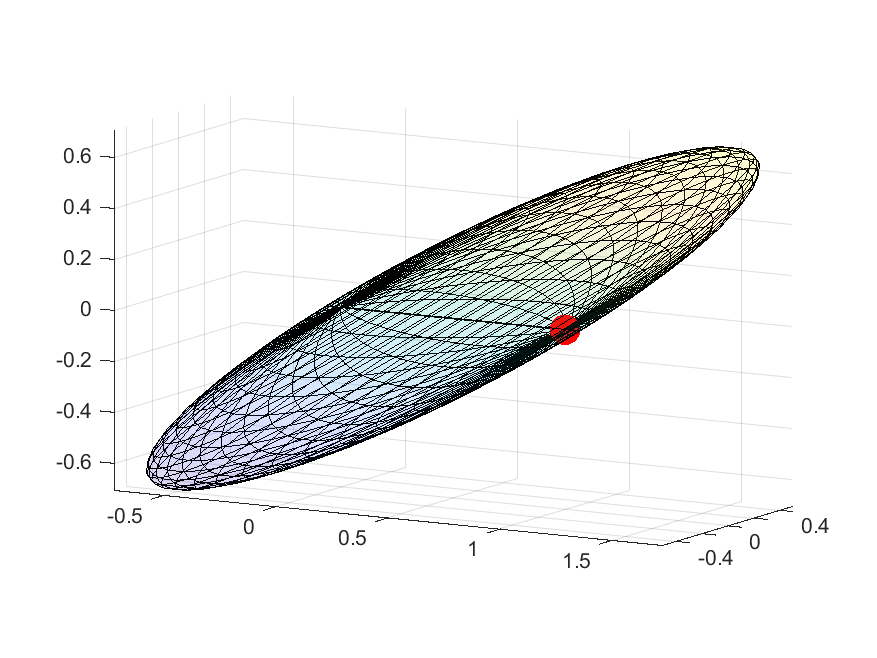}\\
    \includegraphics[width=0.67\linewidth]{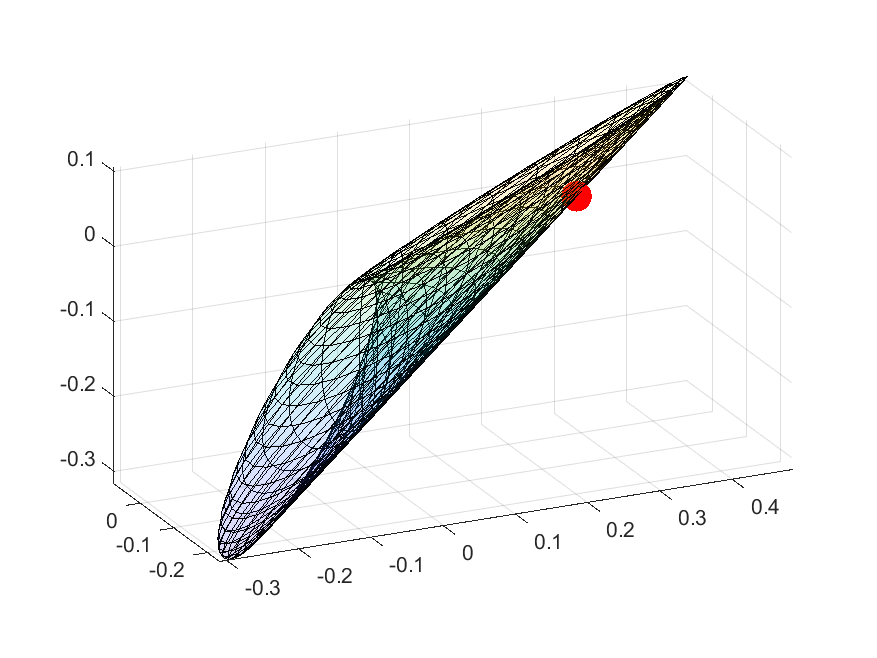}
    \caption{The sets $\mathscr{F}(A_Y,A_1,A_2)\subset\R^{3}$ and the location of the corresponding vectors $\varrho\, e_1\in\R^3$ for the matrices in Examples~\ref{ex:singular} (top), \ref{ex:jordan} (middle), and \ref{ex:LauRhia} (bottom).}\label{fig:kfield-plots}
\end{figure}

\section{Min-max vs. max-min problems}\label{sec:minmax}

Using the submultiplicativity of the Euclidean norm on $\F^n$, we easily see that a lower bound on the value of the best approximation problem~\eqref{eqn:main-prob} is given by 
\begin{equation}\label{eqn:maxmin-minmax}
\max_{\substack{v\in\F^n\\ \|v\|_2=1}}\min_{X\in\mathcal{X}}\|(Y-X)v\|_2 \leq \min_{X\in\mathcal{X}}\|Y-X\|_2=\min_{X\in\mathcal{X}}\max_{\substack{v\in\F^n\\\|v\|_2=1}}\|(Y-X)v\|_2.
\end{equation}
A special case of \eqref{eqn:maxmin-minmax} occurs
in the analysis of the GMRES method for solving linear algebraic systems $Ax=b$. In that context we have $Y=I$ and $\X=\mathrm{span}\{A,\dots,A^k\}\subset\F^{n\times n}$, and the two expressions on the right hand side are equal to the ideal GMRES problem mentioned in the Introduction; see~\eqref{eqn:iG}. The max-min problem that gives the lower bound in \eqref{eqn:maxmin-minmax} is called the \emph{worst-case GMRES problem}. Its value gives an attainable upper bound on the $k$th relative residual norm of the GMRES method applied to $Ax=b$ with $\|b\|_2=1$ and the initial approximation $x_0=0$; see, e.g.,~\cite{FabLieTic13}. We then ask under which conditions ideal and worst-case GMRES are equal. Note that in \cite[equations (2.3) and (2.14)]{LiTi2009} we defined a family of related problems where an analogous question can also be asked.

The paper~\cite{FaJoKnMa1996} contains a detailed analysis of the relation between worst-case and ideal GMRES using the set $\mathscr{F}(A,\dots,A^k)$. 
(The authors refer to ideal GMRES as \emph{polynomial preconditioning}.) In particular, the paper contains one of the first published examples that for these problems a strict inequality may hold in \eqref{eqn:maxmin-minmax} (see, e.g.,~\cite{FabLieTic13,To1997} for further examples), and it is shown in~\cite[Theorems~2.7 and~2.8]{FaJoKnMa1996} that
\begin{align*}
& \max_{\substack{v\in\F^n\\ \|v\|_2=1}}\min_{X\in\mathcal{X}}\|(I-X)v\|_2=1\quad
\Longleftrightarrow
\quad 0\in \mathscr{F}(A,\dots,A^k),\\
& \min_{X\in\mathcal{X}}\max_{\substack{v\in\F^n\\\|v\|_2=1}}\|(I-X)v\|_2=1 \quad
\Longleftrightarrow
\quad 0\in \mathrm{cvx}(\mathscr{F}(A,\dots,A^k)),  
\end{align*}
where $\X=\mathrm{span}\{A,\dots,A^k\}$. Thus, if $\mathscr{F}(A,\dots,A^k)$ is convex, then worst-case and ideal GMRES either both converge, or both stagnate completely until step~$k$. Moreover, if $0\in \mathscr{F}(A,\dots,A^k)$ and $\mathscr{F}(A,\dots,A^k)$ is convex, then equality holds in \eqref{eqn:maxmin-minmax} for $Y=I$ and $\X=\mathrm{span}\{A,\dots,A^k\}$. 

We will now derive  general conditions under which equality holds in \eqref{eqn:maxmin-minmax}. Our investigations use the following well known result, which shows a certain equivalence of optimality and orthogonality; see, e.g.,~\cite[Theorem~2.3.2]{LieStr13}.

\begin{lem}\label{lem:orth}
Consider a $k$-dimensonal subspace $\X=\mathrm{span}\{X_1,\dots,X_k\}\subset \F^{n\times n}$, and let $Y\in \F^{n\times n} \setminus\X$, $X_*\in\X$, and $w\in\F^n$ be given. Then the following are equivalent:
\begin{itemize}
\item[$(1)$] The vector $X_*w$ is a best approximation of $Yw$ with respect to the Euclidean norm, i.e.,
$$\|(Y-X_*)w\|_2=\min_{X\in\X}\|(Y-X)w\|_2.$$
\item[$(2)$] $(Y-X_*)w\perp \mathrm{span}\{X_1w,\dots,X_kw\}$.
\end{itemize}
\end{lem}

The next result contains our necessary and sufficient conditions for equality in~\eqref{eqn:maxmin-minmax}.

\begin{thm}\label{thm:minmax}
Consider a $k$-dimensional subspace $\X=\mathrm{span}\{X_1,\dots,X_k\}\subset \F^{n\times n}$, and let $X_*\in\X$ be a spectral approximation of $Y\in \F^{n\times n} \setminus\X$. In the notation of Theorem~\ref{thm:cond5}, the following are equivalent:
\begin{itemize}
\item[$(1)$] Equality holds in~\eqref{eqn:maxmin-minmax}. 
\item[$(2)$] $0\in \mathscr{F}(A_1,\dots,A_k;\Sigma_{Y-X_{*}})$,
\item[$(3)$] $\varrho\, e_1 \in \mathscr{F}(A_Y, A_1,\dots,A_k)$.
\end{itemize}
\end{thm}

\begin{proof}
$(1) \Rightarrow (2)$: If equality holds in \eqref{eqn:maxmin-minmax}, i.e.,
$$
\|Y-X_*\|_2=\min_{X\in\X}\|Y-X\|_2
=\max_{\substack{v\in\F^n \\ \|v\|_2=1}}
\min_{X\in\X} \|(Y-X)v\|_2,$$
then there exists a unit Euclidean norm vector $w\in\F^n$ such that 
$$
\|Y-X_*\|_2=\min_{X\in\X} \|(Y-X)w\|_2 \leq \| (Y-X_*)w \|_2 \leq  \|Y-X\|_2,
$$ 
and instead of inequalities we can write equalities.
Therefore, $w\in \Sigma_{Y-X_*}$
and Lemma~\ref{lem:orth} yields $(Y-X_*)w\perp \mathrm{span}\{X_{1}w,\dots,X_{k}w\}$, so that 
\begin{equation}\label{eqn:wc1}
0\in\mathscr{F}(A_{1},\dots,A_{k};\Sigma_{Y-X_*}).
\end{equation}

$(2) \Rightarrow (3)$:  If \eqref{eqn:wc1} holds, then there exists a unit Euclidean norm vector $w\in\Sigma_{Y-X_{*}}$ such that $w^{H}A_jw=0$ for $j=1,\dots,k$. Since $X_{*}\in \X$, i.e., $A_{*}\in \mathrm{span}\{A_1,\dots,A_k\}$, this yields $w^{H}A_{*}w=0$, and hence
$$\varrho=w^{H}(A_Y-A_{*})w=w^{H}A_Yw,$$
so that $\varrho\,e_1\in\mathscr{F}(A_Y,A_1,\dots,A_{k})$.

$(3) \Rightarrow (1)$: If $\varrho\,e_1\in\mathscr{F}(A_Y,A_1,\dots,A_{k})$, then
there exists a unit Euclidean norm vector
$w\in\F^n$ such that 
$\varrho = w^H A_Y w$ and
$$
0 = w^{H}A_{j}w = w^{H}(Y-X_{*})^{H}X_{j}w,\quad\mbox{for $j=1,\dots,k$,
}$$ 
i.e., $(Y-X_*)w\perp \mathrm{span}\{X_{1}w,\dots,X_{k}w\}$.
The orthogonality conditions also imply $w^{H}A_{*}w = 0$, 
so that 
$$\varrho=w^{H}A_Yw = w^{H}(A_Y-A_{*})w=w^{H}(Y-X_*)^H (Y-X_*)w,$$
and $w\in\Sigma_{Y-X_*}$. Finally, Lemma~\ref{lem:orth} now yields
\begin{align*}
\min_{X\in\X}\|Y-X\|_2 &= \|Y-X_{*}\|_2=\|(Y-X_{*})w\|_2\\
&=\min_{X\in\mathcal{X}}\|(Y-X)w\|_2\leq \max_{\substack{v\in\F^n\\\|v\|_2=1}}\min_{X\in\mathcal{X}}\|(Y-X)v\|_2,
\end{align*}
so that we have equality in \eqref{eqn:maxmin-minmax}. 
\end{proof}

Note that the equivalence of (1) and (2) in Theorem~\ref{thm:minmax} can also be deduced from \cite[Lemma~2.16]{FaJoKnMa1996}.

%\begin{example}
%We already know that condition (3) in Theorem~\ref{thm:minmax} holds for the matrices considered in Example~\ref{ex:jordan}. In particular, we therefore have
%
%$$\max_{\substack{v\in\F^n\\ \|v\|_2=1}}\min_{\substack{\mathrm{deg}(p)\leq k\\p(0)=1}}\|p(J_n(0))v\|_2=
%\min_{\substack{\mathrm{deg}(p)\leq k\\p(0)=1}}\max_{\substack{v\in\F^n\\ \|v\|_2=1}}\|p(J_n(0))v\|_2=1.$$
%
%This is no surprise, since we know from~\cite[Proposition~2.1]{FaJoKnMa1996} that equality holds in \eqref{eqn:maxmin-minmax} for $Y=I$ and $\X=\mathrm{span}\{A,\dots,A^k\}$ whenever $A$ is singular.
%\end{example}

Combining Theorem~\ref{thm:minmax} with the equivalent conditions \eqref{eqn:cond4} and \eqref{eqn:cond5} 
yields the following result.

\begin{cor}
In the notation of Theorem~\ref{thm:minmax}, suppose that  $X_*\in\X$ is a spectral approximation of $Y\in \F^{n\times n} \setminus\X$. If the set $\mathscr{F}(A_1,\dots,A_k;\Sigma_{Y-X_{*}})$ or the set $\mathscr{F}(A_Y, A_1,\dots,A_k)$ is convex, then equality holds in \eqref{eqn:maxmin-minmax}.
\end{cor}
\begin{proof}
If $X_*$ is a spectral approximation of $Y$, then the equivalent conditions \eqref{eqn:cond4} and \eqref{eqn:cond5}  show that
$$ 
0\in\mathrm{cvx}(\mathscr{F}(A_1,\dots,A_k;\Sigma_{Y-X_{*}})) 
\quad\mbox{and}\quad  
\varrho\, e_1 \in\mathrm{cvx}(\mathscr{F}(A_Y,A_1,\dots,A_k)),
$$
respectively. If $\mathscr{F}(A_1,\dots,A_k;\Sigma_{Y-X_{*}})$
or 
$\mathscr{F}(A_Y,A_1,\dots,A_k)$
is convex, then  Theorem~\ref{thm:minmax} implies that equality holds in \eqref{eqn:maxmin-minmax}.
\end{proof}

From this result we obtain the following sufficient conditions for equality in~\eqref{eqn:maxmin-minmax}.

\begin{cor}\label{cor:cases}
In the notation of Theorem~\ref{thm:minmax}, suppose that  $X_*\in\X$ is a spectral approximation of $Y\in \F^{n\times n} \setminus\X$. We have equality in \eqref{eqn:maxmin-minmax} if one of the following conditions holds:
\begin{itemize}
\item[$(1)$] $k=1$.
\item[$(2)$] $\mathrm{dim}(\Sigma_{Y-X_*})=1$.
\item[$(3)$] $k=2$, $\F=\R$, and $\mathrm{dim}(\Sigma_{Y-X_*})\geq 3$.
\end{itemize}
\end{cor}

\begin{proof}
(1) Let $V\in\F^{n\times\ell}$ be a matrix with orthonormal columns that form a basis of $\Sigma_{Y-X_{*}}$. Any unit Euclidean norm vector $v\in\Sigma_{Y-X_{*}}$ can be written in the form $v=V\xi$, where $\xi\in\F^{\ell}$ has unit Euclidean norm, and we obtain
\begin{align*}
\mathscr{F}(X_{1};\Sigma_{Y-X_*})&=\left\{v^H(Y-X_*)X_1v\,:\, v\in\Sigma_{Y-X_*},\,\|v\|_2=1\right\}\\
&=
\left\{ \xi^{H}V^{H}(Y-X_{*})^{H}X_{1}V\xi\, :\, \xi\in\F^\ell,\;\|\xi\|_2=1\right\}. 
\end{align*}
This is either the real (if $\F=\R$) or classical (if $\F=\C$) field of values of the matrix $V^{H}(Y-X_{*})^{H}X_{1}V\in\F^{\ell\times\ell}$, which in both cases is a convex set.

(2) Let $\mathrm{dim}(\Sigma_{Y-X_*})=1$, so that $\Sigma_{Y-X_*}=\mathrm{span}\{w\}$ for some unit Euclidean norm vector $w\in\F^n$. Any unit Euclidean norm vector $v\in\Sigma_{Y-X_*}$ is of the form $v=e^{i\theta}w$ (if $\F=\C$) or $v=\pm w$ (if $\F=\R$), so that 
$$\mathscr{F}(A_{1},\dots,A_k;\Sigma_{Y-X_*})$$
is a single point in $\F^k$ and hence convex.

(3)  In our notation, Brickman~\cite{Br1961} showed that for any $\ell\geq 3$ and any two symmetric matrices $S_1,S_2\in\R^{\ell\times \ell}$, the set $\mathscr{F}(S_1,S_2)$
is convex. Now let $X_*\in\X=\mathrm{span}\{X_1,X_2\}\subset\R^{n\times n}$ be a spectral approximation of $Y\in\R^{n\times n}$, and suppose that $\ell\equiv\mathrm{dim}(\Sigma_{Y-X_*})\geq 3$. If the columns of $V\in\R^{n\times \ell}$ form an orthonormal basis of  $\Sigma_{Y-X_*}$, then any unit Euclidean norm vector $v\in\Sigma_{Y-X_*}$ can be written as $v=V\xi$ for some unit Euclidean norm vector $\xi\in\R^3$, and 
\begin{align*}
\mathscr{F}(A_{1},A_2;\Sigma_{Y-X_*})
& =  \left\{ \begin{bmatrix}
v^{T}A_{1}v\\
v^{T}A_{2}v
\end{bmatrix}:v\in\Sigma_{Y-X_{*}},\ \|v\|_2=1\right\} \\
 & =  \left\{ \begin{bmatrix}
\xi^{T}B_{1}\xi\\
\xi^{T}B_{2}\xi
\end{bmatrix}:\xi\in\mathbb{R}^{\ell},\ \|\xi\|_2=1\right\}, 
\end{align*}
where $B_{1}\equiv V^{T}A_{1}V$ and $B_{2}\equiv V^{T}A_{2}V$ are real $\ell\times\ell$ matrices. These matrices can be written as $B_j=S_j+N_j$ with $S_j=\frac12 (B_j+B_j^T)$ and $N_j=\frac12 (B_j-B_j^T)$, $j=1,2$. We have $\zeta^TN_j \zeta=0$ for every $\zeta\in\R^\ell$, and therefore
\begin{align*}
\mathscr{F}(A_{1},A_2;\Sigma_{Y-X_*}) =  \left\{ \left[\begin{array}{c}
\xi^{T}S_{1}\xi\\
\xi^{T}S_{2}\xi
\end{array}\right]:\xi\in\mathbb{R}^{\ell},\ \|\xi\|_2=1\right\}
=\mathscr{F}(S_1,S_2),
\end{align*}
which is convex by Brickman's result, since $\ell\geq 3$ and $S_1,S_2\in\R^{\ell\times \ell}$ are symmetric.
\end{proof}

We have some remarks about the three sufficient conditions in Corollary~\ref{cor:cases}:

(1) For $Y=I$ and $\X=\mathrm{span}\{A\}$ the first condition in Corollary~\ref{cor:cases} shows that worst-case and ideal GMRES are equal in iteration $k=1$. This has previously been shown in~\cite[Theorem~1]{Jo1994}; see also~\cite[Theorem~2.5]{GrGu1994} for a closely related result formulated and proven for the space $\R^{n\times n}$.

(2) This condition is a generalization of a similar result for the ideal GMRES problem for $k=1$ in~\cite[Lemma~2.4]{GrGu1994}.

(3) By combining conditions (2) and (3) we see that the only unresolved case 
of whether the equality in \eqref{eqn:maxmin-minmax} holds 
for $k=2$ and $\F=\R$, is when $\mathrm{dim}(\Sigma_{Y-X_*})=2$. In this case
the set $\mathscr{F}(A_{1},A_2;\Sigma_{Y-X_*})$ is either a line segment (i.e., a convex set) or an ellipse. Suppose that $\mathscr{F}(A_{1},A_2;\Sigma_{Y-X_*})$ is an ellipse. From the equivalent condition \eqref{eqn:cond4} in Corollary~\ref{cor:cond4} we know that $0\in\mathrm{cvx}(\mathscr{F}(A_{1},A_2;\Sigma_{Y-X_*}))$. In order to prove that equality in \eqref{eqn:maxmin-minmax} holds also for $k=2$, $\F=\R$ and $\ell=2$, we thus need to prove that zero lies on the ellipse (and not in its interior).

\medskip
Finally, we will prove that equality in~\eqref{eqn:maxmin-minmax} always holds when we ``double'' the problem, i.e., when we consider the best approximation problem
$$\min_{X\in\mathcal{X}}\left\Vert \begin{bmatrix}
Y-X & 0\\
0 & Y-X
\end{bmatrix}\right\Vert _{2}=\min_{X\in\mathcal{X}}\left\Vert I_2 \otimes (Y-X)\right\Vert _{2}.$$
We have used such ``doubled'' problems in the context of worst-case and ideal GMRES also in~\cite[Section~6]{FabLieTic13}. If $X_*\in\X$ is a spectral approximation of $Y\in\F^{n\times n}\setminus\X$, then it is easy to see that
\begin{align*}
\|Y-X_*\|_2 &=\min_{X\in\mathcal{X}}\|Y-X\|_2=
\min_{X\in\mathcal{X}}\left\Vert I_2 \otimes (Y-X)\right\Vert _{2}\\
&=\min_{\substack{X\in\mathcal{X}\\ Z\in\mathcal{X}}}
\left\Vert \begin{bmatrix}
Y & 0\\
0 & Y
\end{bmatrix}-\begin{bmatrix}
X & 0\\
0 & Z
\end{bmatrix}\right\Vert _{2}.
\end{align*}
Using the previous results, we now prove the following theorem.

\begin{thm}
Consider a $k$-dimensional subspace $\mathcal{X}=\mathrm{span}\{X_{1},\dots,X_{k}\}\subset \F^{n\times n}$, and let $X_{*}\in\mathcal{X}$ be a spectral approximation of $Y\in\F^{n\times n}\backslash\mathcal{X}$. Then
\begin{equation}\label{thm:minmaxB}
\max_{\substack{v\in\F^{2n}\\ \|w\|_2=1}}
\min_{X\in\mathcal{X}}\left\Vert \left(I_2 \otimes (Y-X)\right)w\right\Vert _{2}=\min_{X\in\mathcal{X}}\left\Vert I_2 \otimes (Y-X)\right\Vert _{2}.
\end{equation}
\end{thm}

\begin{proof} 
Since $X_{*}$ is a spectral approximation of $Y$, we know by the equivalent characterization \eqref{eqn:cond5} that
$$
\varrho \, e_1 \in 
\mathrm{cvx} \left( \mathscr{F}(A_Y, A_1,\dots,A_k) \right).
$$
Hence, there are two unit Euclidean norm vectors $x\in\mathbb{F}^{n}$
and $y\in\mathbb{F}^{n}$, and a scalar $\alpha\in[0,1]$ such that 
$$
\alpha\begin{bmatrix}
x^{H}A_Yx\\
x^{H}A_1x\\
\vdots\\
x^{H}A_{k}x
\end{bmatrix}+(1-\alpha)\begin{bmatrix}
y^{H}A_Yy\\
y^{H}A_{1}y\\
\vdots\\
y^{H}A_{k}y
\end{bmatrix}=\begin{bmatrix}
\varrho \\
0\\
\vdots\\
0
\end{bmatrix},
$$
which is equivalent to 
$$
\begin{bmatrix}
w^{H} (I_2 \otimes A_Y)w\\
w^{H}(I_2 \otimes A_1)w\\
\vdots\\
w^{H}(I_2 \otimes A_{k})w
\end{bmatrix}
= \begin{bmatrix}
\varrho \\
0\\
\vdots\\
0
\end{bmatrix},
\quad
\mbox{where}
\quad
w \equiv \begin{bmatrix}
\sqrt{\alpha}x\\
\sqrt{1-\alpha}y
\end{bmatrix},\quad \|w\|_2=1.
$$
Therefore,
\[
\varrho\,e_1 \in \mathscr{F}(I_2 \otimes A_Y,I_2 \otimes A_1,\dots,I_2 \otimes A_{k}),
\]
and, using Theorem~\ref{thm:minmax}, we obtain the equality \eqref{thm:minmaxB}.
\end{proof}

\section{Conclusions}
In this work, we revisited the problem of characterizing spectral approximations of a given matrix. In the spirit of the work by Lau and Riha we approached this problem through the general framework of best approximation in normed linear spaces, using Singer’s theorem as a foundation, which we adapted to our specific setting. Our focus was on a geometric characterization of  spectral approximations based on the $k$-dimensional field of $k$ matrices. This approach offered a new perspective that may also prove useful for a better understanding the relationship between the corresponding min-max and max-min approximation problems. Throughout the paper we carefully examined how our results relate to existing work in the literature.

One of our main motivations for studying spectral approximations is the analysis of problems arising in the study of convergence of Krylov subspace methods, as discussed, for example, in \cite{FaJoKnMa1996,FabLieTic09,FabLieTic13,LiTi2009,LieTic14,TiLiFa2007}. In this paper, we considered more general problems that help place our specific questions within a broader mathematical context.

%\bibliographystyle{siam}
%\bibliography{paper}

\end{document}